\documentclass{amsproc}
\usepackage{amsmath, amstext, amsbsy, amssymb}

\newtheorem{theorem}{Theorem}[section]

\newtheorem{proposition}{Proposition}[section]

\theoremstyle{definition}

\theoremstyle{remark}

\numberwithin{equation}{section}

\renewcommand{\a}{\alpha}

\newcommand{\ep}{\epsilon}
\newcommand{\vep}{\varepsilon}

\def\Za{\mathbf Z}
\def\Zp{\bf Z}  

\def\a{\alpha}

\begin{document}

\title[Bosonic realization of toroidal Lie algebras]
{Bosonic realization of toroidal Lie algebras of classical types}
\author{Naihuan Jing$^*$}
\address{Department of Mathematics,
   North Carolina State University,
   Raleigh, NC 27695-8205, USA}
\email{jing@math.ncsu.edu}
\author{Kailash C. Misra}
\address{Misra: Department of Mathematics,
   North Carolina State University,
   Raleigh, NC 27695-8205, USA}
\email{misra@math.ncsu.edu}
\author{Chongbin Xu}
\address{Xu: School of Mathematics and Information Science,
Wenzhou University, Wenzhou 325035, China}
\email{xuchongbin1977@126.com}
\thanks{$*$ Corresponding author.}
\thanks{Jing acknowledges the support of NSA grant H98230-06-1-0083 and NSFC grant
10728102, and Misra acknowledges the support of NSA grant
H98230-08-0080.} 

\begin{abstract}
Generalizing Feingold-Frenkel's construction we use Weyl bosonic
fields to construct toroidal Lie algebras of types $A_n, B_n$, $C_n$
and $D_n$ of level $-1, -2, -1/2$ and $-2$ respectively. In
particular, our construction also gives new bosonic construction for
the orthogonal Lie algebras in the cases of affine Lie algebras.
\end{abstract}

\maketitle

\section{Introduction} \label{S:intro}

Toroidal Lie algebras are natural generalization of the affine
Kac-Moody algebras \cite{MRY} that enjoy many similar interesting
features. Let $\mathfrak{g}$ be a finite-dimensional complex simple
Lie algebra of type $X_n$, and $R=\mathbb{C}[s,s^{-1},t,t^{-1}]$ be
the ring of Laurent polynomials in commuting variables $s$ and $t$.
By definition a 2-toroidal Lie algebra of type $X_n$ is a perfect
central extension of the iterated loop algebra $\mathfrak{g}\otimes
R$, and can be realized as certain homomorphic image of the
universal central extension $T(X_n)=(g\otimes R)\oplus \Omega_R/dR$,
where $\Omega_R/dR$ is K\"ahler differentials of $R$ modulo the
exact forms. The center contains two special elements $c_0, c_1$. A
module of $T(X_n)$ is called a level-$(k_0,k_1)$ module if the
standard pair of central elements $(c_0,c_1)$ acts as $(k_0,k_1)$
for some complex numbers $k_0$ and $k_1$. In this work we will focus
on modules with $k_0\neq 0$.

Two-toroidal Lie algebras resemble affine Lie algebras in many
aspects. Their main source of representations are vertex operators
(cf. \cite{MRY, B, BBS}). In \cite{T} the toroidal Lie algebra of
type $B_n$ was constructed by fermionic operators (see also
\cite{JMg}), and in \cite{G} two constructions were given for the
extended affine Lie algebras of type $A_n$. In \cite{FJW} the first
author and collaborators used McKay correspondence to realize level
one representation of toroidal Lie algebras of simply laced types.
In the recent paper \cite{JM} a unified fermionic construction of
toroidal Lie algebras of classical types was given by extending
Feingold-Frenkel realization coupled with new ghost fields. We
pointed out there that the fermionic construction can not be
directly generalized to symplectic toroidal Lie algebras as one
needs bosons to realize the long root vector $2\vep_n$.

In this paper we give a unified Weyl bosonic construction of all
classical toroidal Lie algebras. Our work in types $C$ and $A$ are
toroidal analog of Feingold-Frenkel construction \cite{FF} for the
affine Lie algebras (see also \cite{L} for a recent treatment of
Feingold-Frenkel construction). We also construct bosonic
realizations for types $D$ and $B$. This includes, as special cases,
new constructions (at level $-2$) for orthogonal affine Lie
algebras. The main idea is again similar to that of \cite{JM} to
construct certain operators corresponding to the special nodes in
the affine Dynkin diagrams. This novelty is special for the cases of
toroidal Lie algebras. In a sense we have also obtained field
operators corresponding to imaginary root vectors.

The Weyl bosonic construction for orthogonal toroidal Lie algebras
is not directly obtained by mimicking that of type $A$ or $C$ and we
have to use new embedding of the Lie algebras into type $A$ to
achieve the goal, this partly explains why this construction has
been missed in previous available constructions of affine Lie
algebras. Like the Neveu-Schwarz and Ramond fields in the fermionic
case, our field operators have two forms indexed by half-integers
and integers respectively.

 The structure of the paper is as follows. In section 2 we define
the toroidal Lie algebra, and state MRY-presentation \cite{MRY} of
the toroidal algebra in terms of generators and relations. In
section 3 we start with a finite rank lattice with an anti-symmetric
bilinear form and define a Fock space and some bosonic field
(vertex) operators, which in turn give level $-1$ representations of
the toroidal Lie algebras of type $A_n$, level $-1/2$ modules for
type $C$, and level $-2$ modules for types $B_n, D_n$. The proof is
an extensive analysis of the operator product expansions for the
field operators. We also include the verification of the Serre
relations.

\section{Toroidal Lie Algebras}

A special quotient algebra of the toroidal Lie algebra $T(X_n)$ is
the double affine algebra, denoted by $T_0(X_n)$, that is the
toroidal Lie algebra of type $X_n$ with a two dimensional center.
The double affine algebra is the quotient of $T(X_n)$ modulo all the
central elements with degree other than zero. In fact, $T_0(X_n)$
has the following realization
$$
T_0(X_n)=(\mathfrak{g}\otimes R)\oplus \mathbb{C}c_0\oplus
\mathbb{C}c_1,
$$
where $R$ is a ring and the Lie product is
$$
[x\otimes g_1, y\otimes g_2]=[x,y]\otimes g_1g_2+\Phi(g_2\partial_sg_1)c_0
+\Phi(g_2\partial_tg_1)c_1
$$
for all $x,y\in \mathfrak{g}$, $g_1,g_2\in R$, where $\Phi$ is a
linear functional on $R$  defined by $\Phi(s^kt^m)=0$, if
$(k,m)\not= (0,0)$ and $\Phi(s^kt^m)=1$, if $(k,m)=(0,0)$ for all
$k,m\in \mathbb{Z}$.

For our purpose we will need the formal power series in variables
$z,w$. In particular we will need the
formal delta functions $\delta(z-w)=\sum_{n\in\mathbb Z}z^{-n-1}w^n$.
This could be understood as follows.
\begin{align*}
\delta(z-w)&=\iota_{z, w}((z-w)^{-1})+\iota_{w, z}((w-z)^{-1}),\\
\partial_w\delta(z-w)&=
\iota_{z, w}((z-w)^{-2})-\iota_{w, z}((w-z)^{-2}),
\end{align*}
where $\iota_{z, w}$ means expansion when $|z|>|w|$. For simplicity
in the following we will drop $\iota_{z, w}$ if it is
clear from the context.

For $n\geq 1$ let $(a_{ij})_{{n+1}\times {n+1}}$ be the generalized Cartan matrix of the
affine algebra $X_n^{(1)}$, and $Q:=
\mathbb{Z}\alpha_0 \oplus \mathbb{Z}\alpha_1  \oplus \cdots \oplus \mathbb{Z}\alpha_n$ its root lattice. The
toroidal Lie algebra $T(X_n)$ \cite{MRY} is the Lie algebra generated
by $\not c$, $\alpha_i(k)$,  and
$x_k(\pm\alpha_i)$ for $i=0,1, \cdots , n$, $k\in \mathbb{Z}$ with the
following relations:

   $(R0)$ $[{\not c}, \alpha_i(z)]=0=[{\not c}, x(\pm\alpha_i,z)];$

   $(R1)$ $[\alpha_i(z),\alpha_j(w)]=
(\a_i | \a_j)\partial_w\delta(z-w){\not c};$

   $(R2)$ $[\alpha_i(z), x(\pm\alpha_j, w)]=
\pm (\a_i | \a_j)x(\pm\alpha_j,w)\delta (z-w);$

   $(R3)$ $[x(\alpha_i,z),x(-\alpha_j,w)]=
\delta_{ij}\frac{2}{(\a_i | \a_j)}\left\{ \alpha_i(w)\delta(z-w)+
\partial_w\delta(z-w){\not c}\right\};$

   $(R4)$ $[x(\alpha_i,z),x(\alpha_i,w)]=0=
[x(-\alpha_i,z),x(-\alpha_i,w)];$

  $(S1)$ $\mathrm{ad}x(\pm\alpha_i,z_1)x(\pm\alpha_j,z_2)=0,
\quad\mbox{for $a_{ij}=0$}$

  $(S2)$ $(\mathrm{ad}x(\pm\alpha_i,z_1))(\mathrm{ad}x(\pm\alpha_i,z_2))
   x(\pm\alpha_j,z_3)=0,$ if $a_{ij}=-1$

  $(S3)$ $(\mathrm{ad}x(\pm\alpha_i,z_1))(\mathrm{ad}x(\pm\alpha_i,z_2))
(\mathrm{ad}x(\pm\alpha_i,z_3))x(\pm\alpha_j,z_4)=0,$ if
$a_{ij}=-2.$ where we have used the follow power series:
 $$
 \alpha_i(z)=\sum_{n\in \mathbb{Z}}\alpha_i(n)z^{-n-1},
 $$
 $$
 x(\pm \alpha_i,z)=\sum_{n\in \mathbb{Z}}x_n(\pm\alpha_i)z^{-n-1}.
 $$
We remark that the 2-toroidal Lie algebra of type $X_n$ includes two
affine Lie algebras of type $X_n^{(1)}$ as subalgebras.

\section{Representations of the Toroidal Algebra}
In this section we give two unified bosonic realizations for the
toroidal Lie algebra of classical types $A_{n-1}, B_n, C_n$ and
$D_n$ using bosonic analogs of Neveu-Schwarz and Ramond fermioninc
fields.

Let $\vep_i$ $(i=1, \ldots, n+1$) be a set of orthonormal basis of
the vector space $\mathbb C^{n+1}$ equipped with the inner product
$( \ | \ )$ such that
\begin{equation}
(\vep_i|\vep_j)=\delta_{ij},
\end{equation}
Let $P_0=\mathbb Z\vep_1\oplus \cdots \oplus \mathbb Z\vep_{n}$ be a
sublattice of rank $n$, and let $\vep_0'$ be an external unit vector
orthogonal to $\vep_i$ and introduce
$\overline{c}=\frac1{\sqrt{2}}(\vep_0'+i\vep_{n+1})$ correspond to
the null vector $\delta$ and $\overline{d}=
\frac1{\sqrt{2}}(\vep_0'-i\vep_{n+1})$ be the dual gradation
operator. Then
\begin{align}
(\overline{c}|\overline{c})&=(\overline{c}|\vep_i)=0\\
(\overline{d}|\overline{d})&=(\overline{d}|\vep_i)=0, \qquad
(\overline{c}|\overline{d})=1
\end{align}
for $i=1, \ldots, n$.

The simple roots for the classical finite dimensional
Lie algebras can be realized
simply by defining the simple roots as follows:

\medskip
$\a_1=\vep_1-\vep_2$, $\cdots$, $\a_{n-1}=\vep_{n-1}-\vep_n$; for $A_{n-1}$.

$\a_1=\vep_1-\vep_2$, $\cdots$, $\a_{n-1}=\vep_{n-1}-\vep_n$,
$\a_n=\vep_n$; for $B_n$.

$\a_1=\frac1{\sqrt 2}(\vep_1-\vep_2)$, $\cdots$,
$\a_{n-1}=\frac1{\sqrt 2}(\vep_{n-1}-\vep_n)$,
$\a_n=\sqrt{2}\vep_n$; for $C_n$.

$\a_1=\vep_1-\vep_2$, $\cdots$, $\a_{n-1}=\vep_{n-1}-\vep_n$,
$\a_n=\vep_{n-1}+\vep_n$; for $D_n$.

\medskip

Then the set of positive roots are:

$$
\Delta_+=\begin{cases} \{\vep_i-\vep_j|1\le i<j \le n\}, & \text{Type $A_{n-1}$}\\
\{\vep_i, \vep_i\pm\vep_j|1 \le i<j \le n\}, & \text{Type $B_n$}\\
\{\sqrt{2}\vep_i, \frac1{\sqrt 2}(\vep_i\pm\vep_j)|1 \le i<j \le n\}, & \text{Type $C_n$}\\
\{\vep_i\pm\vep_j|1 \le i<j \le n\}, & \text{Type $D_n$}.
\end{cases}
$$

The highest (long) root $\a_{max}$ for each type is given as follows:

$$
\a_{max}=\begin{cases} \vep_1-\vep_n, & \text{Type $A_{n-1}$}
\\
\vep_1+\vep_2, & \text{Type $B_n$ or $D_n$}
\\
\sqrt 2\vep_1 & \text{Type $C_n$}
\end{cases}.
$$

We further introduce the element
$$\a_0=\overline{c}-\a_{max}$$
in the lattice and then define $\beta=-\overline{c}+\vep_1$ for type
$ABD$, and $\beta=-\sqrt2\overline{c}+\vep_1$ for type $C$. Then we
have
$$
\a_0=\vep_{n}-\beta\ \ \text{for $A_{n-1}$}; \quad -\beta-\vep_2 \ \
\text{for $B_n, D_n$}; \quad \text{or} -\frac1{\sqrt
2}(\beta+{\vep_1}) \ \ \text{for $C_n$}.
$$
Note that $(\beta|\beta)=1, (\beta|\vep_i)=\delta_{1i}$.

 Then
$P=\mathbb Z\overline{c}\oplus\mathbb Z\vep_1\oplus \cdots \oplus
\mathbb Z \vep_n$ and $Q=\mathbb Z\overline{c}\oplus\mathbb
Z\a_1\oplus \cdots \oplus \mathbb Z \a_k= \mathbb Z\a_0\oplus \cdots
\oplus \mathbb Z \a_k$ ($k= n-1$ for type $A_{n-1}$ and $k= n$ for
types $B_nC_nD_n$) are the weight lattice and root lattice for the
corresponding affine Lie algebra, and $P_0=\mathbb Z\vep_1\oplus
\cdots \oplus \mathbb Z \vep_n$ and $Q_0=\mathbb Z\a_1\oplus \cdots
\oplus  \mathbb Z\a_k$ ($k= n-1$ for type $A_{n-1}$ and $k= n$ for
types $B_nC_nD_n$) are the weight lattice and root lattice for the
simple Lie algebras $A_{n-1}, B_n, C_n, D_n$. Then
$(\a_i|\a_j)=d_ia_{ij}$, where $a_{ij}$ are the entries of the
affine Cartan matrix of type $(ABCD)^{(1)}$, and the $d_i$'s are
given by:
$$
(d_0, d_1, \cdots, d_k) = \begin{cases} \{1, 1, \cdots, 1,1\}, &  k=n-1,\text{Type $A_{n-1}$}\\
\{1, 1, \cdots, 1,\frac{1}{2}\}, &  k=n,\text{Type $B_n$}\\
\{1, \frac{1}{2}, \cdots, \frac{1}{2},1\}, & k=n,\text{Type $C_n$}\\
\{1, 1, \cdots, 1,1\}, &  k=n,\text{Type $D_n$}.
\end{cases}
$$
Note that the bilinear form $(\ \ | \ \ )$ is non-degenerate on the
affine root lattice or the span of $\{\beta\}\cup \{\alpha_i\}$
($i\geq 1$).

We introduce infinite dimensional Weyl algebras as follows. Let $\bf
Z=\mathbb Z$ or ${\bf Z}=\mathbb Z+1/2$. Let $P_{\mathbb
C}={P}\otimes\mathbb C$ be the $\mathbb C$-vector space spanned by
$\overline c$ and $\vep_i, 1 \le i \le n$ in various types. To
combine all types we will also introduce another set of orthonormal
vectors $\vep_{\overline i}, 1 \le i \le n+1$ such that
\begin{equation}
(\vep_{\overline i}|\vep_{\overline j})=\delta_{ij},
\end{equation}
and we denote by $\overline{P}_{\mathbb C}$ the $\mathbb C$-space
spanned by $\vep_{\overline i}$. We also define
$\overline{\beta}=-\overline{c}+\vep_{\overline 1}$, and
$\alpha_{\overline 0}=\overline{c}-\alpha_{\overline{max}}$, where
$\alpha_{\overline{max}}=\vep_{\overline 1}+\vep_{\overline 2}$. The
special vectors are only needed in types $B_n, D_n$ and
$\vep_{\overline{n+1}}$ is only needed for type $B_n$. We define
$\mathcal C=\mathcal C_1\oplus \mathcal C_2$, where both subspaces
$\mathcal C_1=P_{\mathbb C}\oplus \overline{P}_{\mathbb C}$ and
$\mathcal C_2=P_{\mathbb C}^*\oplus \overline{P}_{\mathbb C}^*$ are
maximal isotropic subspaces, where we define $<P, \overline{P}>=0$.
We define the natural anti-symmetric bilinear form on $\mathcal C$
by
\begin{equation}
\langle b^*, a\rangle=-\langle a, b^*\rangle=(a|b),
\quad \langle a, b\rangle=\langle a^*, b^*\rangle=0, \qquad a, b
\in \mathcal C_1
\end{equation}

The Weyl algebra $W(P)$ is the unital algebra generated by the
elements $a(k)$ and $a^*(k)$, where $a \in\mathcal C_1$,
$a^*\in \mathcal C_2$, and $k\in \Za$
subject to the relations:
$$
[u(k), v(l)]=\langle u, v\rangle\delta_{k, -l}, \qquad u, v\in
\mathcal C.
$$

The representation space is the infinite dimensional vector space
$$\displaystyle
V=\bigotimes_{a_i}\left(\bigotimes_{k\in\Zp_+} \mathbb C[a_i(-k)]
\bigotimes_{k\in\Zp_+} \mathbb C[a_i^*(-k)]\right)
$$
where $a_i$ runs through any basis in $P$ and $\overline P$, say
$\overline{c}$, $\vep_i$'s and $\vep_{\overline i}$'s.

The Weyl algebra acts on the space $V$ by the usual action: $a(-k)$
acts as an creation operator and $a(k)$ as an annihilation operator.
We remark that using $-\langle\ , \rangle$ as the form on $\mathcal
C$ is equivalent to switching $\mathcal C_1$ and $\mathcal C_2$.
As the symplectic form $\langle \ \ , \ \ \rangle$ is non-degenerate, the Fock space $V$ is an irreducible
module for the Weyl algebra.

For any bosonic fields
\begin{equation*}
u(z)=\sum_{n\in\Za}u(n)z^{-n-1/2}, \quad
v(z)=\sum_{n\in\Za}v(n)z^{-n-1/2} \end{equation*} we define the
normal ordering $:u(z)v(w):$ by swapping their components:
\begin{equation}
:u(m)v(n):=\begin{cases} u(m)v(n) & m<0\\
\frac12(u(0)v(n)+v(n)u(0)) & m=0\\
v(n)u(m) & m>0\end{cases}.
\end{equation}
It then follows that $:u(z)v(w): \,= :v(w)u(z):$.

Based on the normal product of two fields, we can define the normal
product of $n$ fields inductively as follows. In fact we define that

\begin{equation*}
:u_1(z_1)u_2(z_2)\cdots u_n(z_n):=:u_1(z_1)(:u_2(z_2)\cdots
u_n(z_n):):,
\end{equation*}
and then use induction till we reach two fields.

We define the contraction of two states by
$$
\underbrace{a(z)b(w)}=a(z)b(w)-:a(z)b(w):,
$$
which contains all poles for $a(z)b(w)$. In general, the contraction
of several pairs of states is given inductively by
the following rule.

\begin{theorem} \label{Wick} (Wick's theorem \cite{FF, K}) For elements $v_1, \ldots, v_n$
we have
\begin{align*}
&:u_1\cdots u_m::v_1\cdots v_n:=:u_1\cdots u_mv_1\cdots v_n:+\\
&\sum_{i_1<\cdots<i_s}\underbrace{u_{i_1}v_{j_1}}\cdots
\underbrace{u_{i_s}v_{j_s}} :u_1\cdots \hat{u}_{i_1}\cdots
\hat{u}_{i_s}\cdots u_m v_1\cdots \hat{v}_{j_1}\cdots
\hat{v}_{j_s}\cdots v_n:
\end{align*}
where the sum runs over all possible contractions of some $u_i$'s
and some $v_i$'s.
\end{theorem}

\begin{proposition} \label{OPE} The basic operator product expansions are:
for $x, y\in{\mathcal C}$ we have
$$
\underbrace{u(z)v(w)}=\frac{\langle u, v\rangle}{z-w} \ \ \text{for}
\ \ \mathbb Z+1/2, \quad \langle u,
v\rangle\frac{(z+w)(zw)^{-1/2}}{2(z-w)} \ \ \mbox{for} \ \ \mathbb
Z.
$$
In particular we have for $a, b\in P_{\mathbb C}$ and $\mathbb Z+1/2$
\begin{align*}
\underbrace{a(z)b(w)}&=\underbrace{a^*(z)b^*(w)}=0,\\
\underbrace{a(z)b^*(w)}&=\frac{\langle a, b^*\rangle}{z-w}, \quad
\underbrace{a^*(z)b(w)}=\frac{\langle a^*, b\rangle }{z-w},
\end{align*}
\end{proposition}
\begin{proof}
In fact one has
\begin{align*}
\underbrace{a(z)b^*(w)}&=\frac{\langle a,
b^*\rangle}2z^{-1/2}w^{-1/2}\delta_{\Za, \mathbb Z}+
\sum_{n\in\Za, 0<m\in\Zp}[a(m), b^*(n)]z^{-m-1/2}w^{-n-1/2}\\
&=\langle a, b^*\rangle(\frac12z^{-1/2}w^{-1/2}\delta_{\Za, \mathbb Z}+\sum_{0<m\in\Zp}z^{-m-1/2}w^{m-1/2})\\
&=\begin{cases}\frac{\langle a, b^*\rangle }{z-w} & \Za={\mathbb
Z}+1/2\\
\langle a, b^*\rangle \frac{(z+w)(zw)^{-1/2}}{2(z-w)} & \Za={\mathbb
Z}
\end{cases}.
\end{align*}
The other OPEs are proved in the same manner.
\end{proof}
Note that in both cases ($\mathbb Z$ and $\mathbb Z+1/2$) we have
the following result.
\begin{proposition} \label{commutation} The bosonic fields satisfy the following commutation relations:
\begin{equation*}
[u(z), v(w)]=\langle u, v\rangle\delta(z-w)
\end{equation*}
\end{proposition}
\begin{proof} Using the fact that $:u(z)v(w):=:v(w)u(z):$
it follows from Proposition \ref{OPE} that
\begin{align*}
[u(z), v(w)]&=u(z)v(w)-v(w)u(z)\\
&=\langle u, v\rangle\delta(z-w).
\end{align*}
\end{proof}

\begin{proposition} \label{bracket} The commutators among normal order products are given by:
\begin{align*}
&[:a_1(z)a_2^*(z):,:b_1(w)b_2^*(w):]=\langle a_1, b_2^*\rangle
:a_2^*(z)b_1(z):\delta(z-w)\\ & \qquad + \langle a_2^*, b_1\rangle
:a_1(z)b_2^*(z):\delta(z-w)
+\langle a_1, b_2^*\rangle \langle a_2^*, b_1\rangle\partial_w\delta(z-w).\\
\end{align*}
\end{proposition}
\begin{proof} By Wick's Theorem \ref{Wick} and
Proposition \ref{OPE}, we have for $\mathbb Z+1/2$
\begin{align*}
&:a_1(z)a_2^*(z)::b_1(w)b_2^*(w):=:a_1(z)a_2^*(z)b_1(z)b_2^*(w):\\
&\qquad\qquad +\frac{\langle a_1, b_2^*\rangle
}{z-w}:a_2^*(z)b_1(w):+ \frac{\langle a_2^*, b_1\rangle
}{z-w}:a_1(z)b_2^*(w):+\frac{\langle a_1, b_2^*\rangle\langle a_2^*,
b_1\rangle }{(z-w)^2}.
\end{align*}
Hence the result follows. The case of $\mathbb Z$ is shown
similarly.
\end{proof}

The (anti-symmetric) inner product of the underlying Lie algebra can
be extended to that of the linear factors as follows:
$$\langle:r_1r_2:, :s_1s_2:\rangle= -\langle r_1, s_1\rangle\langle r_2, s_2
\rangle
 +\langle r_1, s_2\rangle\langle r_2, s_1\rangle.$$

 In the following we will describe the bosonic realization of the
toroidal Lie algebra of classical types.

(i) Type $A_{n-1}$ ($n\geq 2$). The roots of the finite dimensional
Lie algebra are $\vep_i-\vep_j$ $(1\leq i\neq j\leq n)$. We
associate to each finite root and imaginary root $\delta$ the
following field operators. The level $\not c=-1$.
\begin{align*}
X(\vep_i-\vep_j, z)&=:\vep_i(z)\vep_j^*(z):,
\quad h_{\vep_i-\vep_j}(z)=:\vep_i(z)\vep_i^*(z):-:\vep_j(z)\ep_j^*(z):\\
X(\pm\alpha_i, z)&=X(\pm(\vep_i-\vep_{i+1}), z), \ \
\alpha_i(z)=h_{\vep_i-\vep_{i+1}}(z), \\
X(\alpha_0, z)&=:\vep_n(z)\beta^*(z):, \quad X(-\alpha_0,
z)=:\beta(z)\vep_n^*(z):,  \\
a_{0}(z)&=:\vep_n(z)\vep_n^*(z):-:\beta(z)\beta^*(z):.\\
\end{align*}
 Using Proposition \ref{bracket} we get for $i\neq j$
\begin{align*}
[X(\vep_i-\vep_j, z), X(\vep_k-\vep_l, w)]
&=\delta_{jk}X(\vep_i-\vep_l, z)\delta(z-w)\\
-\delta_{li}&X(\vep_k-\vep_j, z)\delta(z-w)-\delta_{jk}\delta_{li}\partial_w\delta(z-w),\\
[X(\vep_i-\vep_j, z),X(\vep_j-\vep_i, w)]&=
\delta_{ij}(h_{\vep_i-\vep_j}(z)\delta(z-w)-\partial_w\delta(z-w)),\\
[h_{\vep_i-\vep_j}(z), X(\vep_k-\vep_l,
w)]=&(\vep_i-\vep_j|\vep_k-\vep_l) X(\vep_j-\vep_k, w)\delta(z-w),
\end{align*}
where $\vep_i, \vep_j$ can be $\beta$ (recall
$(\beta|\vep_i)=\delta_{1i}$). These commutation relations show that
the construction provides a level $-1$ realization of the affine Lie
algebra of type $A$. Moreover they also provide a representation for
the toroidal Lie algebra, as the commutation relations involving
$\alpha_0$ are also similar to the above. In fact,
\begin{align*}
[X(\alpha_0, z), X(\alpha_1, w)]&=:\vep_n(z)\vep^*_2(z):\delta(z-w),
\quad [X(\alpha_0, z), X(-\alpha_1,
w)]=0,\\
[X(\alpha_0, z), X(\alpha_n,
w)]&=-:\beta^*(z)\vep_{n-1}(z):\delta(z-w), \quad [X(\alpha_0, z),
X(-\alpha_n, w)]=0.
\end{align*}

(ii) Type $C_n$ ($n\geq 2$). The finite roots are
$\alpha_{ij}=\frac1{\sqrt{2}}(\vep_i-\vep_j)$ $(1\leq i\neq j\leq
n)$, $\pm\beta_{ij}= \pm\frac1{\sqrt{2}}(\vep_i+\vep_j)$ $(1\leq
i\neq j\leq n)$. Note that $-\alpha_{ij}=\alpha_{ji}$. We realize
the toroidal Lie algebra by the following field operators and the
level is $\not c=-1/2$.
\begin{align*}
X(\alpha_{ij}, z)&=:\vep_i(z)\vep_j^*(z):, \quad
h_{\vep_i-\vep_j}(z)=:\vep_i(z)\vep_i^*(z):-:\vep_j(z)\vep_j^*(z):,\\
X(\beta_{ij}, z)&=:\vep_i(z)\vep_j(z):, \qquad
X(-\beta_{ij}, z)=-:\vep_i^*(z)\vep_j^*(z):,\\
h_{\vep_i+\vep_j}(z)&=:\vep_i(z)\vep_i^*(z):+:\vep_j(z)\vep_j^*(z):,\\
X(\pm\alpha_i, z)&=X(\pm\alpha_{i, i+1}, z), \ \
\alpha_i(z)=(1/2)h_{\vep_i-\vep_{i+1}}(z),\\
X(\pm\alpha_n, z)&=(1/2) X(\pm\beta_{nn}, z),
\quad \alpha_n(z)=h_{2\vep_n}(z),\\
X(\alpha_0, z)&=(1/2):\beta^*(z)\vep_1^*(z):, \quad X(-\alpha_0,
z)=(1/2):\beta(z)\vep_1(z): , \\
\alpha_{0}(z)&=(1/4)(:\vep_1(z)\vep_1^*(z):+:\beta(z)\beta^*(z):
+:\beta^*(z)\vep_1(z):+:\vep_1^*(z)\beta(z):).\\
\end{align*}
Several commutation relations are exactly the same as type A. Some
new ones are obtained by Proposition \ref{bracket} as follows.
\begin{align*}
[X(\beta_{ij},& z), X(-\beta_{kl}, w)]\\
&=(\delta_{ik}X(\alpha_{jl}, z)+\delta_{il}X(\a_{jk},
z)+\delta_{jk}X(\alpha_{jl}, z)\\
&+\delta_{jl}X(\alpha_{ik},
z))\delta(z-w)-(\delta_{ik}\delta_{jl}+\delta_{ik}\delta_{jl})\partial_w\delta(z-w),\\
[X(\alpha_{ij},& z), X(\beta_{kl}, w)]=(\delta_{jk}X(\beta_{il},
z)+\delta_{jl}X(\beta_{ik}, z))\delta(z-w),\\
[X(\alpha_{ij},& z), X(-\beta_{kl}, w)]=-(\delta_{ik}X(-\beta_{jl},
z)+\delta_{il}X(-\beta_{jk}, z))\delta(z-w).
\end{align*}

(iii) Type $D_n$ ($n\geq 4$). The finite roots are $\pm \vep_i\pm
\vep_j$ ($1\leq i\neq j\leq n$). We realize the toroidal Lie algebra
by the following field operators. The level $\not c=-2$. For $i<j$
we define
\begin{align*}
X(\vep_i-\vep_j, z)&=:\vep_i(z)\vep_j^*(z):-
:\vep_{\overline j}(z)\vep_{\overline i}^*(z):, \\
h_{\vep_i-\vep_j}(z)&=:\vep_i(z)\vep_i^*(z):-:\vep_j(z)\vep_j^*(z):
 -:\vep_{\overline i}(z)\vep_{\overline i}^*(z):+
:\vep_{\overline j}(z)\vep_{\overline j}^*(z):,\\
X(\vep_i+\vep_j, z)&=:\vep_i(z)\vep_{\overline j}^*(z):-
:\vep_j(z)\vep_{\overline i}^*(z):,\\
X(-\vep_i-\vep_j, z)&=:\vep_{\overline j}(z)\vep_i^*(z):
-:\vep_{\overline i}(z)\vep_{j}^*(z):, \\
h_{\vep_i+\vep_j}(z)&=:\vep_i(z)\vep_i^*(z):+:\vep_j(z)\vep_j^*(z):-:\vep_{\overline
i}(z)\vep_{\overline i}^*(z):-
:\vep_{\overline j}(z)\vep_{\overline j}^*(z):,\\
X(\pm \alpha_i, z)&=X(\pm(\vep_{i}-\vep_{i+1}), z), \ \ \alpha_i(z)=h_{\vep_i-\vep_{i+1}}(z),\\
X(\pm \alpha_n, z)&=X(\pm(\vep_{n-1}+\vep_n), z), \ \
\alpha_n(z)=h_{\vep_{n-1}-\vep_{n}}(z),\\
 X(\alpha_0,
z)&=:\vep_{\overline 2}(z)\beta^*(z):
-:\overline{\beta}(z)\vep_{2}^*(z):,\\
X(-\alpha_0,z)&=:\beta^*(z)\vep_{\overline 2}^*(z):
-:\vep_{2}(z)\overline{\beta}(z):,\\
\alpha_{0}(z)&=:\vep_{\overline 2}(z)\vep_{\overline 2}^*(z):
-:\vep_{2}(z)\vep_{2}^*(z):+:\overline{\beta}(z)\overline{\beta}^*(z):
-:{\beta}^*(z){\beta}(z):.
\end{align*}

(iv) Type $B_n$($n\geq 2$). The finite roots are $\pm \vep_i\pm
\vep_j$ $(1\leq i\neq j\leq n)$ and $\pm \vep_i$ $(1\leq i\leq n)$.
The level $\not c=-2$. We realize the toroidal Lie algebra by the
following field operators. The field operators for the roots
$\pm\vep_i\pm\vep_j$ ($i\neq j$) are the same as in type $D_n$,
though we need to add the bosonic fields $\vep_{\overline{n+1}}(z)$
and $\vep_{\overline{n+1}}^*(z)$ in the realization.
\begin{align*}
X(\vep_i, z)&=:\vep_i(z)\vep_{\overline{n+1}}^*(z):-
:\vep_{\overline{n+1}}(z)\vep_{\overline i}^*(z):,\\
X(-\vep_i, z)&=
:\vep_{\overline{n+1}}(z)\vep_i^*(z):
-:\vep_{\overline i}(z)\vep_{\overline{n+1}}^*(z):,\\
h_{\vep_i}(z)&=:\vep_i(z)\vep_i^*(z):-:\vep_{\overline
i}(z)\vep_{\overline i}^*(z):, \\
X(\pm\alpha_n, z)&=X(\pm\vep_n, z), \quad \alpha_n(z)=\vep_n(z),
\end{align*}
and $X(\pm\alpha_i, x)$, $\alpha_i(z)$ ($0\leq i\leq n-1$) are given
by the same formulae as in type $D_n$.  Using the same analysis we
find that
\begin{align*}
[X(\vep_i, z), X(-\vep_j, w)]&=X(\vep_i-\vep_j,
z)\delta(z-w)-2\delta_{ij}\partial_w\delta(z-w),\\
[X(\vep_i, z), X(\vep_j, w)]&=-X(\vep_i+\vep_j, z)\delta(z-w).
\end{align*}

\begin{theorem} The previous constructions
give rise realizations of classical toroidal Lie algebras of types
$A_{n-1} (n\geq 2), C_n (n\geq 2), D_n (n\geq 4), B_n (n\geq 3)$
 with level $-1, -1/2, -2, -2$ respectively.
\end{theorem}

\begin{proof} In the above construction we have outlined major
commutation relations of various types. One can use $d_i$ to combine
the argument for types $A_{n-1}$ and $C_n$. We will denote
$\alpha_{ij}=d_i^{-1/2}(\vep_i-\vep_j)=-\alpha_{ji}$ and
$\beta_{ij}=d_i^{-1/2}(\vep_i+\vep_j)$ for $i\neq j$, then both
types have the finite positive simple $\alpha_{i, i+1}$ plus
$\beta_{n,n}$ in type $C_n$. First we note that for $i\neq j, k\neq
l$
\begin{align} \notag
[:\vep_i(z)&\vep_j^*(z):, :\vep_k(w)\vep_l^*(w):]=
(\delta_{jk}:\vep_i(z)\vep_{l}^*(z):   \\
&-\delta_{il}:\vep_{j}^*(z)\vep_{k}(z):)
\delta(z-w)-\delta_{il}\delta_{jk}\partial_w\delta(z-w).
\label{E:com-A}
\end{align}
When $i=l$, $j=k$ the identity reduces to
\begin{equation*}
[X(\a_{ij}, z), X(-\a_{ij},
w)]=d_i^{-1}(a_{ij}(z)\delta(z-w)-d_i\partial_w\delta(z-w))
\end{equation*}
which shows that the center $\not c=-d_i$ for type $A_{n-1}$ and
$C_n$. Here we also put
\begin{equation} \label{E:root-A}
a_{ij}(z)=d_i(:\vep_i(z)\vep_i^*(z):-:\vep_j(z)\vep_j^*(z):).
\end{equation}
Then the Heisenberg field operators are generated by the simple ones
$\a_i(z)=a_{i, i+1}(z)$. We then have by using (\ref{E:com-A})
\begin{align*}
[\a_i(z),\a_j(w)]&=-d_i^2(2\delta_{ij}-\delta_{i+1,j}-\delta_{i,j+1})\partial_w\delta(z-w)\\
&=-d_i(\a_i | \a_j)\partial_w\delta(z-w).
\end{align*}
Also for any $k\neq l$ we find that (using \ref{E:com-A})
\begin{align} \notag
&[\alpha_i(z), X(\a_{kl}, w)]\\
&=[d_i(:\vep_i(z)\vep_i^*(z):-:\vep_{i+1}(z)\vep_{i+1}^*(z):),
:\vep_k(w)\vep_l^*(w)] \notag\\
&=(\a_i|\a_{kl}) X(\alpha_{kl}, z)\delta(z-w) \label{E:com-A2}
\end{align}
Serre relations are consequences of (\ref{E:com-A}) and
(\ref{E:com-A2}) for non-special nodes. It then follows that the
operators $X(\pm\a_i, z)$ generate a type $A$ subalgebra for $i=1,
\ldots, n-1\; (n-2 \; \text{for} \; A_{n-1})$. The verification for
special node $0$ is quite similar to the above provided that we view
$\beta$ as $\vep_1$ in type $C_n$, and $\beta$ as $\vep_n$ in type
$A$ when checking commutation relations.

In types $D_n$ and $B_n$, the basic commutation relations are
already given above using similar analysis to Eq. (\ref{E:com-A}).
For instance in type $B_n$ we have that
\begin{align*}
&[X(\a_0, z), X(-\a_1, w)]\\
&=[:\vep_0(z)\vep_1^*(z):-:\vep_{\overline 1}(z)\vep_{\overline
0}^*(z):, :\vep_2(w)\vep_1^*(w):-:\vep_{\overline
1}(w)\vep_{\overline 2}^*(w):]=0.
\end{align*}

Serre relations can be easily checked from the above commutation
relations plus properties of roots. We include one example to show
the method. In the case of type $C_n$ we have
\begin{align*}
[&X(\alpha_0, z_1), [X(\alpha_0, z_2),X(\alpha_1,
w)]]=\frac12[X(\alpha_0,z_1), [:\beta^*(z_2)\vep_1^*(z_2):,
:\vep_1(w)\vep_2^*(w)]]\\
&=\frac14[:\beta^*(z_2)\vep_1^*(z_2):,
:\beta^*(z_2)\vep_2^*(z_2):+:\vep_1^*(z_2)\vep_2^*(z_2):]\delta(z_2-w)=0.
\end{align*}
\begin{align*}
[&X(\alpha_1, z_1), [X(\alpha_1, z_2),[X(\alpha_1, z_3), X(\alpha_0,
w)]]]\\
&=\frac12[X(\alpha_1,z_1), [X(\alpha_1,
z_2),[:\vep_1(z_3)\vep_2^*(z_3), :\beta^*(w)\vep_1^*(w):
]]]\\
&=-\frac12[X(\alpha_1,z_1), [:\vep_1(z)\vep_2^*(z_2):
,[:\vep_2^*(w)\beta^*(w):+:\vep_2^*(w)\vep_1^*(w):]]\delta(z_3-w)\\
&=[X(\alpha_1,z_1),
:\vep_2^*(w)\vep_2^*(w):]\delta(z_2-w)\delta(z_3-w)=0.
\end{align*}
\end{proof}

The toroidal Lie algebra of type $A_1$ is also realized by the same
formulae as given above. However some commutation relations are
different from $A_{n-1} (n\geq 3)$.
\begin{align*}
[X(\alpha_0, z), &X(\alpha_1,
w)]=(:\vep_2(w)\vep_2^*(w):-:\beta^*(w)\vep_1(w):)\delta(z-w)-\partial_w\delta(z-w),\\
[X(-\alpha_0, z), &X(-\alpha_1,
w)]=(:\beta(w)\vep_1^*(w):-:\vep_2(w)\vep_2^*(w):)\delta(z-w)-
\partial_w\delta(z-w),\\
[X(\alpha_0, z_1),& [X(\alpha_0, z_2), X(\alpha_1, w)]]
=-2:\beta^*(w)\vep_2(w):\delta(z_1-w)\delta(z_2-w),\\
[X(\alpha_1, z_1), &[X(\alpha_1, z_2), X(\alpha_0, w)]]
=-2X(\alpha_1, w)\delta(z_1-w)\delta(z_2-w),
\end{align*}
where the last two relations imply the Serre relations immediately
for toroidal Lie algebra of type $A_1$. We remark that the
construction in type $B_2$ gives another representation of toroidal
Lie algebra of type $C_2$ at level $-2$.

\bibliographystyle{amsalpha}

\end{document}